\documentclass[a4paper,11pt,makeidx]{amsart}
\usepackage{amscd}
\usepackage{xypic}  %commu 
\usepackage{amssymb}
\usepackage{amsthm}
\usepackage{epsf}
\makeindex
\newtheoremstyle{break}% name
  {9pt}%      Space above, empty = `usual value'
  {9pt}%      Space below
  {\itshape}% Body font
  {}%         Indent amount (empty = no indent, \parindent = para indent)
  {\bfseries}% Thm head font
  {.}%        Punctuation after thm head
  {\newline}% Space after thm head: \newline = linebreak
  {}%         Thm head spec
\theoremstyle{break}
\newtheorem{bthm}{Theorem}

\theoremstyle{plain}
\newtheorem{thm}{Theorem}[section]

\newtheorem{lemma}[thm]{Lemma}

\newtheorem{prop}[thm]{Proposition}

\newtheorem{defn}[thm]{Definition}

\newtheorem{rem}[thm]{Remark}

\def\RR{{\mathbb R}}

\def\Im{\operatorname{Im}}

\def\Supp{\operatorname{Supp}}
\def\Exc{\operatorname{Exc}}

\def\trdeg{\operatorname{tr.deg}}
\def\CC{{\mathbb C}}
\def\ZZ{{\mathbb Z}}
\def\NN{{\mathbb N}}
\def\R{{\mathbf R}}
\def\QQ{{\mathbb Q}}
\def\PP{{\mathbb P}}
\def\B{\mathbf{B}}

\def\N{{\mathbf N}}

\def\*{\otimes}
\def\+{\oplus}                   % direct sum
\def\*{\otimes}                  % tensor product
\def\Supp{\operatorname{Supp}}

\def\Bs{\operatorname{Bs}}
\def\vol{\operatorname{vol}}

\begin{document}

\title[Augmented base loci and restricted volumes on normal varieties, II]{Augmented base loci and restricted volumes on normal varieties, II: The case of real divisors}

\author[A.F. Lopez]{Angelo Felice Lopez*}

\thanks{* Research partially supported by the MIUR national project ``Geometria delle variet\`a algebriche" PRIN 2010-2011.}

\address{\hskip -.43cm Dipartimento di Matematica e Fisica, Universit\`a di Roma
Tre, Largo San Leonardo Murialdo 1, 00146, Roma, Italy. e-mail {\tt lopez@mat.uniroma3.it}}


\begin{abstract}
Let $X$ be a normal projective variety defined over an algebraically closed field and let $Z$ be a subvariety. Let $D$ be an $\RR$-Cartier $\RR$-divisor on $X$. Given an expression $(\ast) \ D \sim_{\RR} t_1 H_1 + \ldots + t_s H_s$ with $t_i \in \RR$ and $H_i$ very ample, we define the $(\ast)$-restricted volume of $D$ to $Z$ and we show that it coincides with the usual restricted volume when $Z \not\subseteq \B_+(D)$. Then, using some recent results of Birkar \cite{bir}, we generalize to  $\RR$-divisors the two main results of \cite{bcl}: The first, proved for smooth complex projective varieties by Ein, Lazarsfeld, Musta{\c{t}}{\u{a}}, Nakamaye and Popa, is the characterization of $\B_+(D)$ as the union of subvarieties on which the $(\ast)$-restricted volume vanishes; the second is that $X - \B_+(D)$ is the largest open subset on which the Kodaira map defined by large and divisible $(\ast)$-multiples of $D$ is an isomorphism. 
\end{abstract}

\maketitle

\section{Introduction}
\label{intro}
Let $X$ be a projective variety and let $D$ be an $\RR$-Cartier $\RR$-divisor on $X$. After their introduction in \cite{n, elmnp1}, the stable base loci of $D$ have gained substantial importance in the study of the birational geometry of $X$, see for example \cite{t,hm,bdpp,bchm}, to mention only a few. Let us recall here their definitions.
\begin{defn}
The {\rm stable base locus} of $D$ is
\[ \B(D)= \bigcap_{E \geq 0 : E \sim _{\RR} D} {\rm Supp}(E). \]
The {\rm augmented base locus} of $D$ is 
\[ \B_+(D)= \bigcap\limits_{\genfrac{}{}{0pt}{}{E \geq 0 :  D - E}{\mbox{\rm is ample}}} {\rm Supp}(E) \]
where $E$ is an $\RR$-Cartier $\RR$-divisor. 
\end{defn}
Since $\B_+(D)$ measures the failure of $D$ to be ample, it is clearly a key tool in several instances. On the other hand it is often not so easy to identify. To this end an important result of Ein, Lazarsfeld, Musta{\c{t}}{\u{a}}, Nakamaye,  and Popa \cite[Thm.\ C]{elmnp2} comes to help, at least when $X$ is complex and smooth:
\[ \B_+(D)= \bigcup\limits_{\genfrac{}{}{0pt}{}{Z\subseteq X :}{\vol_{X|Z}(D)=0}} Z \]
where, given a subvariety $Z \subseteq X$ of dimension $d > 0$, when $D$ is Cartier (whence also when it is a $\QQ$-Cartier $\QQ$-divisor), one defines the restricted linear series $H^0(X|Z,mD)$ to be the image of the restriction map $H^0(X,mD) \to H^0(Z, mD_{|Z})$ and the restricted volume as 
\[ \vol_{X|Z}(D) = \limsup\limits_{m \to + \infty} \frac{h^0(X|Z, mD)}{m^d/d!}. \]
One of the deep parts of \cite{elmnp2} is then to prove the strong continuity result that, if $Z$ is an irreducible component of $\B_+(D)$, then $\lim\limits_{D' \to D} \vol_{X|Z}(D') = 0$, where $D'$ is a $\QQ$-Cartier $\QQ$-divisor whose class goes to the class of $D$ \cite[Thm.\ 5.7]{elmnp2}.

In \cite[Thm.\ B]{bcl} we generalized \cite[Thm.\ C]{elmnp2} to any normal projective variety defined over an arbitrary algebraically closed field, but for $\QQ$-Cartier $\QQ$-divisors $D$.
This was achieved, in part, by outlining the importance of the behavior on $Z$ of the maps $\Phi_m : X \dashrightarrow \PP H^0(X, mD)$ associated to $mD$.

Now, when $D$ is an $\RR$-Cartier $\RR$-divisor, several difficulties arise, as for example one does not have neither the linear series $|mD|$, nor the associated maps nor, in general, the  restricted volume. Of course one could use the integer part, but, at least for this type of problems, this does not appear to be the right choice (see also Section \ref{vol}).

On the other hand, in a recent article of Birkar, a different approach was taken, and this proved quite successful, since it allowed him to generalize Nakamaye's theorem (\cite[Thm.\ 0.3]{n}, \cite[\S 10.3]{l}) to nef $\RR$-Cartier $\RR$-divisors on arbitrary projective schemes over a field \cite[Thm.\ 1.3]{bir}.

\begin{defn}
\label{esp}
Let $X$ be a projective variety and let $D$ be an $\RR$-Cartier $\RR$-divisor on $X$. We can write
\[ (\ast) \ \ D \sim_{\RR} t_1 H_1 + \ldots + t_s H_s \]
where, for $1 \leq i \leq s$, $H_i$ is a very ample Cartier divisor on $X$ and $t_i \in \RR$. For $m \in \NN$ we set
\[ \langle mD \rangle =  \lfloor mt_1 \rfloor H_1 + \ldots +  \lfloor mt_s \rfloor H_s \]
and
\[ \Phi_{\langle mD \rangle} : X \dashrightarrow \PP H^0(X,\langle mD \rangle). \]
\end{defn}
It is clear that both $H^0(X,\langle mD \rangle)$ and $\Phi_{\langle mD \rangle}$ depend on the expression $(\ast)$. On the other hand, as proved by Birkar, several facts about $D$ are in fact independent on $(\ast)$, such as for example that $D$ is big if and only if the upper growth of $h^0(X, \langle mD \rangle)$ is like $m^{\dim X}$ \cite[Lemma 4.2]{bir} (see also Lemma \ref{maxopen}(i)) or that $\B_+(D)= \bigcap\limits_{m \in \NN} \B(\langle mD \rangle - A)$ can be defined as is done for Cartier divisors \cite[Lemma 3.1]{bir}.

Continuing on this vein, we propose an analogous definition for the restricted volume

\begin{defn}
\label{volristr}
Let $X$ be a projective variety and let $D$ be an $\RR$-Cartier $\RR$-divisor on $X$ with an espression $(\ast)$ as in Definition {\rm \ref{esp}}. Let $Z \subseteq X$ be a subvariety of dimension $d > 0$. We set
\[ \vol_{X|Z}(D, (\ast)) = \limsup\limits_{m \to + \infty} \frac{h^0(X|Z,  \langle mD \rangle)}{m^d/d!}. \]
\end{defn}

and for the stable base locus

\begin{defn}
\label{b'}
Let $X$ be a projective variety and let $D$ be an $\RR$-Cartier $\RR$-divisor on $X$ and fix an espression $(\ast)$ as in Definition {\rm \ref{esp}}.
We set
\[ \B(D, (\ast)) = \bigcap\limits_{m \in \NN^+} \Bs |\langle mD \rangle|. \]
\end{defn}

As we will see in Proposition \ref{volume}, while in general $\vol_{X|Z}(D, (\ast))$ does depend on $(\ast)$, in the important case of $Z \not \subseteq \B_+(D)$, it is independent on $(\ast)$ and coincides with the usual  $\vol_{X|Z}(D)$. 

Our point is that this definition of restricted volume allows to generalize, to any $\RR$-Cartier $\RR$-divisor, the main results of \cite{bcl}.

First, the description of the complement of  $\B_+(D)$ in terms of the maps $\Phi_{\langle mD \rangle}$

\begin{bthm} 
\label{max}
Let $D$ be a big $\RR$-Cartier $\RR$-divisor on a normal projective variety $X$ defined over an algebraically closed field and fix an espression $(\ast)$ as in Definition {\rm \ref{esp}}. Then the complement $X - \B_+(D)$ of the augmented base locus is the largest Zariski open subset $U \subseteq X - \B(D, (\ast))$ such that, for all large and divisible $m$, the restriction of the morphism $\Phi_{\langle mD \rangle}$ to $U$ is an isomorphism onto its image. 
\end{bthm}

Second, the description of $\B_+(D)$ in terms of restricted volume

\begin{bthm} 
\label{b+}
Let $D$ be an $\RR$-Cartier $\RR$-divisor on a normal projective variety $X$ defined over an algebraically closed field and fix an espression $(\ast)$ as in Definition {\rm \ref{esp}}. For every irreducible component $Z$ of $\B_+(D)$ we have $\vol_{X|Z}(D, (\ast))=0$, and hence
\[ \B_+(D)= \bigcup\limits_{\genfrac{}{}{0pt}{}{Z\subseteq X :}{\vol_{X|Z}(D, (\ast))=0}} Z \]
\end{bthm}

\section{Volume and restricted volume of real divisors} 
\label{vol}
Throughout the paper we work over an arbitrary algebraically closed field $k$. An \emph{algebraic variety} is by definition an integral separated scheme of finite type over $k$. 

We set $\NN^+ = \{n \in \NN : n > 0\}$ and, given $x \in \RR$, $\{x\} = x - \lfloor x \rfloor$.

We will often use the following fact, proved in \cite[Thm.\ 1.3]{bir}. Birkar's theorem is deeper, as it proves Nakamaye's theorem (\cite[Thm.\ 0.3]{n}, \cite[\S 10.3]{l}) on arbitrary projective schemes over a field (not necessarily algebraically closed), namely that, if $D$ is nef, then $\B_+(D)$ coincides with the exceptional locus of $D$. On the other hand, for the part of  \cite[Thm.\ 1.3]{bir} that we use, in the proof given in \cite{bir}, the nefness of $D$ is not needed.

\begin{thm}
\label{bi13}
Let $X$ be a projective scheme over a field and let $D$ be an $\RR$-Cartier $\RR$-divisor on $X$ with an espression $(\ast)$ as in Definition {\rm \ref{esp}}. Let $H$ be a very ample Cartier divisor on $X$. Then there exists $m_0 \in \NN^+$ such that $\B_+(D) =  \B(\langle km_0 D \rangle - H) = \Bs |\langle km_0D \rangle - H|$ for all $k \in \NN^+$.
\end{thm}

Let $X$ be a projective variety and let $D$ be an $\RR$-Cartier $\RR$-divisor on $X$ with an espression $(\ast)$ as in Definition \ref{esp}. We start by defining a graded ring associated to $D$ and $(\ast)$.

\begin{rem} 
\label{formuletta}
For $a, b \in \NN^+, t_i \in \RR, 1 \leq i \leq s$, set $\gamma_i(a,b ) = \lfloor (a + b)t_i \rfloor -   \lfloor at_i \rfloor - \lfloor bt_i \rfloor$. Then
\begin{itemize}
\item[(i)] $\gamma_i(a, b) = 0, 1$ for all $i, a, b$; 
\item[(ii)] $\gamma_i(a, b) + \gamma_i(a+b, c) = \gamma_i(a, b+c) + \gamma_i(b, c)$ for all $i, a, b, c$; 
\item[(iii)] $ \langle (a + b) D \rangle =  \langle aD \rangle + \langle bD \rangle + \sum\limits_{i = 1}^s \gamma_i(a,b) H_i$.
\end{itemize}
\end{rem}
By (iii), choosing some divisors $E_i \in |H_i|$, we get a multiplication map
\begin{equation}
\label{molt}
H^0(X, \langle aD \rangle) \otimes H^0(X, \langle bD \rangle) \to H^0(X, \langle (a + b) D \rangle)
\end{equation}
and by (ii) this gives rise to a ring and to a semigroup.
\begin{defn}
\label{ring}
Let $X$ be a projective variety and let $D$ be an $\RR$-Cartier $\RR$-divisor on $X$ with an espression $(\ast)$ as in Definition {\rm \ref{esp}}. Let $E_i \in |H_i|$ for $1 \leq i \leq s$.
The associated graded ring is
\[ \R(X, D, (\ast)) = \bigoplus_{m \in \NN} H^0(X,  \langle mD \rangle) \]
and the associated semigroup is
\[ \N(X, D, (\ast)) = \{m \in \NN : H^0(X,  \langle mD \rangle) \neq 0 \}. \]
Let $Z \subseteq X$ be a subvariety and pick $E_i \in |H_i|$ so that $Z \not\subseteq \Supp(E_i)$ for all $i$. We define 
\[ H^0(X|Z, \langle mD \rangle) = \Im \{H^0(X,  \langle mD \rangle) \to H^0(Z,  \langle mD \rangle_{|Z})\}, \] 
\[ \R(X|Z, D, (\ast)) = \bigoplus_{m \in \NN} H^0(X|Z,  \langle mD \rangle) \] 
and 
\[ \N(X|Z, D, (\ast)) = \{m \in \NN : H^0(X|Z,  \langle mD \rangle) \neq 0 \}. \]
\end{defn}
\begin{rem}
\label{dip}
Note that $\N(X|Z, D, (\ast))$ and $\R(X|Z, D, (\ast))$ depend on $(\ast)$.
{\rm For example let $H$ be a very ample Cartier divisor on $X$ and let $H_i \in |H|, i = 1, 2$ with $H_1 \neq H_2$. Let $\alpha \in \RR - \QQ$ and let $D = \alpha H_1 - \alpha H_2$. If we use this expression as $(\ast)$ we get, for every $m \in \NN^+$, $\langle mD \rangle =  \lfloor m\alpha \rfloor H_1 + \lfloor - m\alpha \rfloor H_2 \sim_{\ZZ} - H$, whence  $\N(X, D, (\ast)) = \{0\}$ and $\R(X, D, (\ast)) = \CC$. But if we use as $(\ast)$ the expression $D \sim_{\RR} 0 H_1 + 0 H_2$ we get $\langle mD \rangle = 0$ and then $\N(X, D, (\ast)) = \NN$ and $\R(X, D, (\ast)) = \bigoplus_{m \in \NN} \CC$.}
\end{rem}

\begin{rem}
\label{homog1}
Let $\sigma_0, \ldots, \sigma_r \in H^0(X, \langle aD \rangle)$ for some $a \in \NN^+$. It is easy to prove that if $\sigma_0^{i_0} \cdot \ldots \cdot \sigma_r^{i_r}$ is a homogeneous product
of degree $d$ in $H^0(X, \langle daD \rangle)$ as in \eqref{molt}, then 
\[ \sigma_0^{i_0} \cdot \ldots \cdot \sigma_r^{i_r} = \sigma_0^{i_0} \ldots \sigma_r^{i_r} \prod\limits_{i=1}^s \varepsilon_i^{\lfloor dat_i \rfloor - d \lfloor at_i \rfloor} \]
where $\varepsilon_i \in H^0(X, H_i)$ is the section defining $E_i$ and the product on the right hand side is the usual product of sections of line bundles.
\end{rem}

Since we have a graded ring structure on $\R(X|Z, D, (\ast))$, it follows that the associated volume function is homogeneous of degree $\dim Z$. For completeness, we give a proof of this fact, in analogy with \cite[Lemma 3.2]{dp}. 

\begin{lemma} 
\label{homog}
Let $D$ be an $\RR$-Cartier $\RR$-divisor on a projective variety $X$ of dimension $n$ and fix an espression $(\ast)$ as in Definition {\rm \ref{esp}}. Let $Z \subseteq X$ be a subvariety of dimension $d > 0$. Then $\vol_{X|Z}(D, (\ast))$ is homogeneous of degree $d$, that is, for every $p \in \NN^+$,
\begin{itemize}
\item[(i)]  $\limsup\limits_{m \to + \infty} \frac{h^0(X,  \langle mD \rangle)}{m^n/n!} = \limsup\limits_{m \to + \infty} \frac{h^0(X,  \langle pmD \rangle)}{(pm)^n/n!}$; 
\item[(ii)] $\limsup\limits_{m \to + \infty} \frac{h^0(X|Z,  \langle mD \rangle)}{m^d/d!}= \limsup\limits_{m \to + \infty} \frac{h^0(X|Z,  \langle pmD \rangle)}{(pm)^d/d!}$.
\end{itemize}
\end{lemma}
\begin{proof} 
Since (i) is just the case $Z = X$ of (ii), let us prove (ii). We can assume that $\N(X|Z, D, (\ast)) \neq \{0\}$. Let $e = e(\N(X|Z, D, (\ast)))$ be the exponent of $\N(X|Z, D, (\ast))$. Then there is $r_0 \in \NN^+$ such that for every $r \geq r_0$ we have that $er \in \N(X|Z, D, (\ast))$ and moreover for every $m \in \N(X|Z, D, (\ast))$ we have that $e | m$. Then
\begin{equation}
\label{uno}
\limsup\limits_{m \to + \infty} \frac{h^0(X|Z,  \langle mD \rangle)}{m^d/d!} = \limsup\limits_{k \to + \infty} \frac{h^0(X|Z,  \langle ekD \rangle)}{(ek)^d/d!}.
\end{equation}
Now let $b = l.c.m.\{e, p\}$ so that $e = vb, p = ab$ with $l.c.m.\{v,a\} = 1$. If $pm \in  \N(X|Z, D, (\ast))$ then $e|pm$, whence $pm = eak$ for some $k \in \NN^+$ and therefore
\begin{equation}
\label{due}
\limsup\limits_{m \to + \infty} \frac{h^0(X|Z,  \langle pmD \rangle)}{(pm)^d/d!} = \limsup\limits_{k \to + \infty} \frac{h^0(X|Z,  \langle eakD \rangle)}{(eak)^d/d!}.
\end{equation}
Let $r \in \NN^+$ be such that $r_0 \leq r \leq r_0 + a$. Then we can find $D_r \in |\langle erD \rangle|$ such that $Z \not\subseteq \Supp(D_r)$ and, for every $q \in \NN^+$ such that $qa- (r_0 + a) \geq r_0$, we can find $D'_r \in |\langle e(qa-r)D \rangle|$ such that $Z \not\subseteq \Supp(D'_r)$. By Remark \ref{formuletta}(iii) we deduce that, for every $k \in \NN^+$,
\[ h^0(X|Z, \langle eakD \rangle) \leq h^0(X|Z, \langle e(ka+r)D \rangle) \leq h^0(X|Z, \langle e(k+q)aD \rangle). \] 
Now exactly as in the proof of \cite[Lemma 2.2.38]{l} we get that
\[ \limsup\limits_{k \to + \infty} \frac{h^0(X|Z,  \langle ekD \rangle)}{(ek)^d/d!} = \limsup\limits_{k \to + \infty} \frac{h^0(X|Z,  \langle eakD \rangle)}{(eak)^d/d!} \]
whence (ii) by \eqref{uno} and \eqref{due}.
\end{proof}

Let $D$ be an $\RR$-Cartier $\RR$-divisor on a projective variety $X$ and let $Z \subseteq X$ be a subvariety of dimension $d > 0$. How to define $\vol_{X|Z}(D)$? When $Z \not \subseteq \B_+(D)$ this is done in \cite{elmnp2}, as follows. Consider the cone ${\rm Big}^Z(X)^+_{\QQ}$ of divisor classes $\xi \in N^1(X)_{\QQ}$ such that $Z \not \subseteq \B_+(\xi)$. By \cite[Thm.\ A]{elmnp2} $\vol_{X|Z}$ is defined on ${\rm Big}^Z(X)^+_{\QQ}$ and extends uniquely to a continuous function on ${\rm Big}^Z(X)^+_{\RR}$. On the other hand if  $Z \subseteq \B_+(D)$ several problems arise, perhaps the most important one being the loss of continuity, as there exist examples \cite[Ex.\ 5.10]{elmnp2} of $\QQ$-Cartier $\QQ$-divisors $D_i$ such that $\lim\limits_{i \to \infty} D_i = D$ but $\vol_{X|Z}(D) \neq \lim\limits_{i \to \infty} \vol_{X|Z}(D_i)$. One possibility to go around this problem is to define, as in \cite[Def.\ 2.12]{leh}, $\vol_{X|Z}$ using the integer parts, that is 
\[ \vol_{X|Z}(D) = \limsup\limits_{m \to + \infty} \frac{h^0(X|Z,  \lfloor mD \rfloor)}{m^d/d!}. \]
We want to point out here that this definition does not agree with $\vol_{X|Z}(D, (\ast))$, and, even more, it can happen that one is zero and the other one is not, as in the following example. 

\begin{rem} (\cite[Ex.\ 5.10]{elmnp2})
\label{ese}
Let $R \subset \PP^3$ be a line and let $\pi : X \to \PP^3$ be the blowing up of $R$ with exceptional divisor $E$. Let $H$ be a plane in $\PP^3$ not containing $R$ and let ${\widetilde H}$ be its strict trasform on $X$. Let $\alpha \in \RR^+ - \QQ$ and $D = \alpha {\widetilde H}$. Let $C$ be a curve of type (2,1) on $E \cong \PP^1 \times \PP^1$. We claim that $\limsup\limits_{m \to + \infty} \frac{h^0(X|C,  \lfloor mD \rfloor)}{m} = \alpha$ while there exists an expression $(\ast)$ as in Definition {\rm \ref{esp}} such that $\vol_{X|C}(D, (\ast)) = 0$.

\noindent {\rm To see this first notice that, as in \cite[Ex.\ 5.10]{elmnp2}, we have that $h^0(X|C,  \lfloor mD \rfloor) = h^0(X|C, \lfloor m\alpha \rfloor {\widetilde H}) = \lfloor m\alpha \rfloor + 1$ and therefore 
\[ \limsup\limits_{m \to + \infty} \frac{h^0(X|C,  \lfloor mD \rfloor)}{m} = \alpha. \]
Now let $A = a {\widetilde H} - E$ for $a \gg 0$ so that $A$ is ample and let $s \gg 0$ be such that $H_2 := sA$ and $H_1 := {\widetilde H} + sA$ are very ample. Then we have the expression 
\[ (\ast) \ \ D = \alpha H_1 - \alpha H_2 \]
as in Definition \ref{esp} and $\langle mD \rangle =  \lfloor m\alpha \rfloor {\widetilde H} - H_2 = (\lfloor m\alpha \rfloor -sa) {\widetilde H} + sE$. But now either $\Bs |\langle mD \rangle| = X$ or $\Bs |\langle mD \rangle| = E$, so that, for all $m \in \NN^+$ we have $C \subset \Bs |\langle mD \rangle|$, whence $h^0(X|C, \langle mD \rangle) = 0$ and $\vol_{X|C}(D, (\ast)) = 0$.}
\end{rem}

This type of phenomenon does not happen when $Z \not \subseteq \B_+(D)$:

\begin{prop} 
\label{volume}
Let $D$ be an $\RR$-Cartier $\RR$-divisor on a projective variety $X$ of dimension $n$ and fix an espression $(\ast)$ as in Definition {\rm \ref{esp}}. Then 
\begin{itemize}
\item[(i)] $vol(D) = \limsup\limits_{m \to + \infty} \frac{h^0(X,  \langle mD \rangle)}{m^n/n!}$; 
\item[(ii)] For every subvariety $Z \subseteq X$ of dimension $d > 0$ such that $Z \not \subseteq \B_+(D)$ we have $vol_{X|Z}(D) = \vol_{X|Z}(D, (\ast))$. 
\end{itemize}
\end{prop}

\begin{proof} For $1 \leq i \leq s$ let $q_{il}, q'_{il}$, $l \in \NN^+$ be two sequences of rational numbers such that $q_{il} \leq t_i, q'_{il} \geq t_i$ for all $l$ and $\lim\limits_{l \to +\infty} q_{il} = \lim\limits_{l \to +\infty} q'_{il} = t_i$. Set $D_l = \sum\limits_{i = 1}^s q_{il} H_i, D'_l = \sum\limits_{i = 1}^s q'_{il} H_i$, so that $D_l$ and $D'_l$ converge to $D$ in $N^1(X)_{\RR}$. Pick $E_i \in |H_i|$. Then, for every $m \in \NN^+$, we have
\[ \langle mD_l \rangle = \sum\limits_{i = 1}^s \lfloor mq_{il} \rfloor H_i \sim_{\ZZ} \sum\limits_{i = 1}^s \lfloor mq_{il} \rfloor E_i \leq \sum\limits_{i = 1}^s \lfloor mt_i \rfloor E_i \sim_{\ZZ} \sum\limits_{i = 1}^s \lfloor mt_i \rfloor H_i =  \langle mD \rangle \]
and similarly
\[ \langle mD \rangle \sim_{\ZZ} \sum\limits_{i = 1}^s \lfloor mt_i \rfloor E_i \leq \sum\limits_{i = 1}^s \lfloor mq'_{il} \rfloor E_i \sim_{\ZZ}  \langle mD'_l \rangle \]
so that
\[ h^0(X, \langle mD_l \rangle) \leq h^0(X, \langle mD \rangle) \leq h^0(X, \langle mD'_l \rangle). \]
Let $p_l, p'_l \in \NN^+$ be such that $p_l D_l$ and $p'_l D'_l$ are Cartier. Note that $m p_l D_l = \langle mp_lD_l \rangle$ and $m p'_l D'_l = \langle mp'_lD'_l \rangle$. Then Lemma \ref{homog}(i) gives
\[ \vol(D_l) = \frac{1}{p_l^n} \vol(p_l D_l) = \limsup\limits_{m \to + \infty} \frac{h^0(X,  \langle mp_lD_l \rangle)}{(mp_l)^n/n!} = \limsup\limits_{m \to + \infty} \frac{h^0(X,  \langle mD_l \rangle)}{m^n/n!} \leq \]
\[ \leq \limsup\limits_{m \to + \infty} \frac{h^0(X,  \langle mD \rangle)}{m^n/n!} \leq \limsup\limits_{m \to + \infty} \frac{h^0(X,  \langle mD'_l \rangle)}{m^n/n!} = \limsup\limits_{m \to + \infty} \frac{h^0(X,  \langle mp'_lD'_l \rangle)}{(mp'_l)^n/n!} = \]
\[ = \frac{1}{(p'_l)^n} \vol(p'_l D'_l) = \vol(D'_l) \]
By \cite[Cor.\ 2.2.45]{l} we get that
\[ \vol(D) = \lim\limits_{l \to + \infty}  \vol(D_l) \leq \limsup\limits_{m \to + \infty} \frac{h^0(X,  \langle mD \rangle)}{m^n/n!} \leq  \lim\limits_{l \to + \infty}  \vol(D'_l) = \vol(D) \]
and this proves (i). To see (ii) we follow the above proof but now choose $E_i \in |H_i|$ such that $Z \not\subseteq \Supp(E_i)$ for all $i$. Then, for $m \in \NN^+$, it follows as above that
\[ h^0(X|Z, \langle mD_l \rangle) \leq h^0(X|Z, \langle mD \rangle) \leq h^0(X|Z, \langle mD'_l \rangle) \]
and Lemma \ref{homog}(ii) gives
\[ \vol_{X|Z}(D_l) \leq \vol_{X|Z}(D, (\ast)) \leq \vol_{X|Z}(D'_l). \]
By \cite[Thm.\ 5.2(a)]{elmnp2} (we note that the proof works on any a projective variety defined over an algebraically closed field) we get as above that
\[ \vol_{X|Z}(D) = \lim\limits_{l \to + \infty}  \vol_{X|Z}(D_l) \leq \vol_{X|Z}(D, (\ast)) \leq  \lim\limits_{l \to + \infty}  \vol_{X|Z}(D'_l) = \vol_{X|Z}(D) \]
and this proves (ii).
\end{proof}

\section{Stable and augmented base loci} 
\label{stab}

Let $X$ be a projective variety and let $D$ be an $\RR$-Cartier $\RR$-divisor on $X$ with an espression $(\ast)$ as in Definition \ref{esp}. We will study stable and augmented base loci associated to $D$ in terms of $(\ast)$, in particular $\B(D, (\ast))$ (see Definition \ref{b'}).

Note that $\B(D, (\ast))$ depends on $(\ast)$. In fact in the example of Remark \ref{dip} we get $\B(D, (\ast)) = X$ when we use $D = \alpha H_1 - \alpha H_2$ and $\B(D, (\ast)) =\emptyset$ when we use $D \sim_{\RR} 0 H_1 + 0 H_2$.

Nevertheless this stable base locus is in between $\B(D)$ and $\B_+(D)$.
\begin{lemma} 
\label{b'stab}
Let $D$ be an $\RR$-Cartier $\RR$-divisor on a projective variety $X$ and fix an espression $(\ast)$ as in Definition {\rm \ref{esp}}. Then
\begin{itemize}
\item[(i)]  there exists $m_0 \in \NN^+$ such that $\B(D, (\ast)) = \Bs |\langle km_0D \rangle|$ for all $k \in \NN^+$;
\item[(ii)] $\B(D) \subseteq \B(D, (\ast)) \subseteq \B_+(D)$;
\item[(iii)] there exists $m_1 \in \NN^+$ such that $\B_+(D) = \B_+(\langle km_1D \rangle)$ for all $k \in \NN^+$.
\end{itemize}
\end{lemma}
\begin{proof} 
To see (i), as in the case of the stable base locus of a Cartier divisor \cite[Prop.\ 2.1.21]{l}, it is enough to notice that, by Remark \ref{formuletta}(iii), it follows that $\Bs |\langle lmD \rangle| \subseteq \Bs |\langle mD \rangle|$, for every $m, l \in \NN^+$. As for the first inclusion in (ii), let $x \in \B(D)$, let $E \in |\langle mD \rangle|$ and let $E_i \in |H_i|$ such that $x \not\in \Supp(E_i)$ for all $i$. Then $F := \frac{1}{m}(E + \sum\limits_{i = 1}^s \{mt_i\} E_i) \sim_{\RR}   \frac{1}{m} \sum\limits_{i = 1}^s mt_i H_i \sim_{\RR} D$ and $F \geq 0$, whence $x \in \Supp(F)$ and then $x \in \Supp(E)$. Therefore $x \in \Bs |\langle mD \rangle|$. 

Now let $H$ be a very ample Cartier divisor on $X$. By Theorem \ref{bi13} there exists $m_1 \in \NN^+$ such that $\B_+(D) =  \B(\langle km_1D \rangle - H) = \Bs |\langle km_1D \rangle - H|$ for all $k \in \NN^+$. Hence by (i) there is $m_2 \in \NN^+$ such that $\B(D, (\ast)) =  \Bs |\langle m_2D \rangle|$ and  $\B_+(D) =  \Bs |\langle m_2D \rangle - H|$, whence $\B(D, (\ast)) \subseteq \B_+(D)$. Finally note that $km_1 D - \langle km_1D \rangle \sim_{\RR} \sum\limits_{i = 1}^s \{km_1t_i\} H_i$ is zero or ample, whence
\[ \B_+(D) = \B_+(km_1D) \subseteq  \B_+(\langle km_1D \rangle) \subseteq \B(\langle km_1D \rangle - H) = \B_+(D) \]
and this gives (iii).
\end{proof}
Note that (iii) above is not needed in the sequel. We just put as it could be useful to know.

We now consider the behavior of the maps associated to $D$ and $(\ast)$.

\begin{lemma}
\label{maxopen}
Let $D$ be an $\RR$-Cartier $\RR$-divisor on a projective variety $X$ and fix an espression $(\ast)$ as in Definition {\rm \ref{esp}}. Then
\begin{itemize}
\item[(i)]  $D$ is big if and only if there exists $m_0 \in \NN^+$ such that $\Phi_{\langle mD \rangle} : X \dashrightarrow \PP H^0(X,\langle mD \rangle)$ is birational onto its image for every $m \geq m_0$.
\item[(ii)] Assume that $D$ is big. For every $m \in \NN^+$ such that $\Phi_{\langle mD \rangle}$ is birational onto its image, let $U_m(\ast) \subseteq X - \B(D, (\ast))$ be the largest open subset on which $\Phi_{\langle mD \rangle}$ is an isomorphism. Then the set $\{U_m(\ast) \}$ has a unique maximal element $U_D(\ast)$, that is there exists $m_0 \in \NN^+$ such that $U_D(\ast) = U_{km_0}(\ast)$ for all $k \in \NN^+$. 
\end{itemize}
\end{lemma}
\begin{proof}
To see (i) assume that $D$ is big and let $A$ be a sufficiently ample Cartier divisor such that $A +  \langle D \rangle$ is globally generated. By Theorem \ref{bi13} there exists $m_1 \in \NN^+$ such that $\B_+(D) =  \Bs |\langle m_1D \rangle - A|$, so that there is $E \in |\langle m_1D \rangle - A|$. Then $\langle m_1D \rangle \sim_{\ZZ} A + E$ and therefore $m_1 \in \N(X, D, (\ast))$. By Remark \ref{formuletta}(iii) there is an effective Cartier divisor $F$ on $X$ such that $\langle (m_1+1) D \rangle \sim_{\ZZ} \langle m_1 D \rangle + \langle D \rangle + F$. Hence $\langle (m_1+1) D \rangle \sim_{\ZZ} A + \langle D \rangle + E + F$ and therefore $m_1+1 \in \N(X, D, (\ast))$. Then $\N(X, D, (\ast))$ has exponent $1$ and there is $r_0 \in \NN^+$ such that $H^0(X, \langle rD \rangle) \neq 0$ for every $r \geq r_0$. Now for every $m \geq m_0:= m_1 + r_0$, we get, by Remark \ref{formuletta}(iii), that we can write $\langle mD \rangle \sim_{\ZZ} \langle m_1D \rangle + \langle (m-m_1)D \rangle + H$ with $H$ zero or very ample. Hence $\langle mD \rangle \sim_{\ZZ} A + E + E' + H$ for some effective Cartier divisor $E'$ and therefore $\Phi_{\langle mD \rangle}$ is an isomorphism over $X - \Supp(E \cup E')$. On the other hand if $\Phi_{\langle mD \rangle}$ is birational onto its image for some $m$, then $\langle mD \rangle$ is big and so is $D$ since $mD \sim_{\RR} \langle mD \rangle + \sum_{i = 1}^s \{mt_i\} H_i$.

To see (ii) note that, for all $m \in \NN^+$ such that $\Phi_{\langle mD \rangle}$ is birational onto its image and for all $k \in \NN^+$, we have by Remark \ref{formuletta}(iii) that also $\Phi_{\langle kmD \rangle}$ is birational onto its image, whence $U_m(\ast) \subseteq U_{km}(\ast) $. If $Y_m(\ast) = X - U_m(\ast)$ we then have $Y_m(\ast) \supseteq Y_{km}(\ast)$ for all $k \in \NN^+$, whence there is a unique minimal element $Y_{m_0}(\ast) = Y_{km_0}(\ast)$ for all $k \in \NN^+$ and therefore a unique maximal element $U_D(\ast) = U_{km_0}(\ast)$ for all $k \in \NN^+$. 
\end{proof}
As a matter of fact we will prove below that $U_D(\ast)$ is independent of $(\ast)$.

\section{Proof of the main theorems} 
\label{proofs}

We follow the proofs in \cite{bcl}.

\subsection{Proof of Theorem \ref{max}}
\begin{proof}
Let $H$ be a very ample Cartier divisor on $X$. By Theorem \ref{bi13} there exists $m_1 \in \NN^+$ such that $\B_+(D) =  \Bs |\langle km_1D \rangle - H|$ for all $k \in \NN^+$. It follows that $|\langle km_1D \rangle - H|$ is base-point free on $X - \B_+(D)$ for all $k \in \NN^+$, which implies that $\Phi_{\langle km_1D \rangle}$ is an isomorphism on $X - \B_+(D)$. Set $U_m = U_m(\ast)$ and $U_D = U_D(\ast)$ (cfr. Lemma \ref{maxopen}(ii)). By Lemmas \ref{b'stab}(ii) and \ref{maxopen}(ii) we get that  $X - \B_+(D) \subseteq U_D$.

Conversely, by Lemmas \ref{b'stab}(i) and \ref{maxopen}(ii), there is an $m_0 \in \NN^+$ such that $\B(D, (\ast)) = \Bs |\langle km_0D \rangle|$ and $U_D = U_{km_0}$ for every $k \in \NN^+$. Set $m=km_0$ and consider the commutative diagram
\begin{equation}
\label{diag}
\xymatrix{ \hskip .2cm X_m \ar_{\mu_m}[d] \ar^{f_m}[r] &  \hskip .1cm Y_m  \ar_{\nu_m}[d] \\ X \ar@{-->}_{\hskip -.4cm \Phi_{\langle mD \rangle}}[r] &  \hskip .1cm \Phi_{\langle mD \rangle}(X) }
\end{equation}
where $\mu_m$ is the normalized blow-up of $X$ along the base ideal of $|\langle mD \rangle|$, $\nu_m$ is the normalization of $\Phi_{\langle mD \rangle}(X)$, and $f_m:X_m\to Y_m$ is the induced birational morphism between normal projective varieties. By construction, we have a decomposition 
\[ \mu_m^*({\langle mD \rangle}) = f_m^*A_m+F_m \]
where $A_m$ is an ample line bundle on $Y_m$ and $F_m$ is an effective divisor with 
\[ \Supp(F_m) = \mu_m^{-1}\left(\B(D, (\ast))\right). \] 
Now $\Phi_{\langle mD \rangle}$ is an isomorphism on $U_m$, whence, by \eqref{diag}, $\nu_m\circ f_m$ is an isomorphism on $\mu_m^{-1}(U_m)$ since $\mu_m$ is an isomorphism over $X-  \B(D, (\ast))$, and it follows that 
\begin{equation}
\label{equ:um}
\mu_m^{-1}(U_m)\subseteq X_m - \left(\Exc(f_m)\cup\Supp(F_m)\right).
\end{equation}
Since $\mu_m(\Exc(\mu_m)) \subseteq \B(D, (\ast))\subseteq \B_+(D)$, by Lemma \ref{b'stab}(ii), using \cite[Prop.\ 2.3]{bbp} (which holds over any algebraically closed field) we get
\[ \B_+(\mu_m^{\ast}D) =  \mu_m^{-1}(\B_+(D)). \]
Now let $G_m =  \sum\limits_{i = 1}^s \{mt_i\} H_i$, so that $G_m$ is zero or ample and $mD \sim_{\RR} \langle mD \rangle + G_m$. Let 
\[ V_m =  \begin{cases} \emptyset & {\rm if} \ G_m = 0 \\ \B_+(\mu_m^*G_m) & {\rm if} \ G_m  \ \mbox{is ample} \end{cases}. \]
By \cite[Prop.\ 2.3]{bbp} we get that, if $G_m$ is ample, then $V_m \subseteq  \Exc(\mu_m) \subseteq \Supp(F_m)$. 
Therefore, by the numerical invariance of $\B_+(D)$ \cite[Prop.\ 1.4]{elmnp1}, we get 
\[ \B_+(\mu_m^*D) = \B_+(\mu_m^*(mD)) = \B_+(\mu_m^*(\langle mD \rangle) + \mu_m^*G_m) = \B_+(f_m^*A_m + F_m + \mu_m^*G_m) \subseteq \] \[ \subseteq \B_+(f_m^*A_m + F_m) \cup V_m \subseteq \B_+(f_m^*A_m) \cup \Supp(F_m). \]
Another application of \cite[Prop.\ 2.3]{bbp} gives
\[ \B_+(f_m^*A_m) \cup \Supp(F_m) \subseteq \Exc(f_m) \cup \Supp(F_m), \]
so that
\begin{equation}
\label{altra}
\mu_m^{-1}\left(\B_+(D)\right) \subseteq \Exc(f_m) \cup \Supp(F_m)
\end{equation}
and, thanks to (\ref{equ:um}), we conclude as desired that $U_m \subseteq X - \B_+(D)$.
\end{proof}

\subsection{Proof of Theorem \ref{b+}}
\label{c}
\begin{proof}
If $D$ is not big Theorem \ref{b+} follows by Proposition \ref{volume}(i). 

Now assume that $D$ is big. We use the notation in the previous subsection. 

Let $Z$ be an irreducible component of $\B_+(D)$, let $d = \dim Z$, so that necessarily $d > 0$ by \cite[Prop.\ 1.1]{elmnp2} (which relies on a result of \cite{z} valid for normal varieties over any algebraically closed field). If $Z \subseteq \B(D, (\ast))$ then  obviously $H^0(X|Z,\langle lD \rangle) = 0$ for every $l \in \NN^+$ and therefore $\vol_{X|Z}(D, (\ast)) = 0$. 

We may thus assume that $Z \not\subseteq \B(D, (\ast))$. The proof of Theorem \ref{max} gives, by \eqref{equ:um} and \eqref{altra}, that, for all $m=km_0$,
\[ \mu_m^{-1}\left(\B_+(D)\right) = \mu_m^{-1}(X - U_m) = \Exc(f_m) \cup \Supp(F_m) \]
so that the strict transform $Z_m$ of $Z$ on $X_m$ is an irreducible component of $\Exc(f_m)$. Since $f_m$ is a birational morphism between normal varieties, it follows that $\dim f_m(Z_m)< \dim Z$, whence, by \eqref{diag}, also that
\[ \dim\Phi_{\langle mD \rangle}(Z) = \dim (\nu_m\circ f_m)(Z_m) < d. \]
As in \cite[Proof of Cor. 2.5]{bcl} we have $\Phi_{\langle mD \rangle}(Z) = \Phi_{W_m}(Z)$, where $W_m = H^0(X|Z,  \langle mD \rangle)$. Hence, setting $\kappa= \kappa(\R(X|Z, D, (\ast)) := \trdeg(\R(X|Z, D, (\ast))-1$ we see, as in \cite[Lemma 2.3]{bcl} (here we use Remark \ref{homog1} - see \cite[Lemma 3.14]{b}), that $\kappa < d$. Now, as in \cite[Prop.\ 2.1]{bcl} (using also \cite[Lemma 4.1]{bir}), we get that there exists $C > 0$ such that $h^0(X|Z,  \langle lD \rangle) \leq C l^{\kappa}$  for every $l \in \NN^+$ and therefore $\vol_{X|Z}(D, (\ast)) = 0$.
 
It remains to prove that, if $Z \subseteq X$ is a subvariety of dimension $d > 0$ such that $Z \not \subseteq \B_+(D)$, then $\vol_{X|Z}(D, (\ast)) > 0$. Since $\B_+(D) = \Bs |\langle m_1D \rangle - H|$ there is $E \in |\langle m_1D \rangle - H|$ such that $Z \not \subseteq \Supp(E)$. Let $l_0 \in \NN^+$ be such that $H^1(X, {\mathfrak I}_Z(lH)) = 0$ for every $l \geq l_0$, so that the commutative diagram
\begin{equation*}
\xymatrix{ \hskip .2cm H^0(X, lH) \ar@{^{(}->}[d]  \ar@{->>}[r] & H^0(Z, lH_{|Z})  \ar@{^{(}->}[d]  \\ H^0(X, l \langle m_1D \rangle) \ar[r] & H^0(Z, l \langle m_1D \rangle_{|Z})  }
\end{equation*}
shows that $h^0(X|Z, l \langle m_1D \rangle) \geq h^0(Z, lH_{|Z}) \geq C l^d$ for some $C > 0$. By Remark \ref{formuletta}(iii) we have that $h^0(X|Z, \langle l m_1D \rangle) \geq h^0(X|Z, l \langle m_1D \rangle)$ and we conclude by Lemma \ref{homog}(ii).
\end{proof}

\end{document}